\documentclass{amsart}
\usepackage{amsthm}
\usepackage{lmodern}
\usepackage[english]{babel}
\usepackage{amsmath}
\usepackage{amssymb}
\usepackage{mathptmx} 
\usepackage{color}
\usepackage{graphicx}
\usepackage{caption}
\usepackage{subcaption}
\usepackage{float}
\usepackage[all,cmtip]{xy}
\usepackage{bm}
\usepackage[a4paper, left=3cm, right=2cm]{geometry}
\usepackage{pinlabel}
\usepackage{enumitem}
\newtheorem{theorem}{\bf Theorem}[section]

\newtheorem{definition}[theorem]{\bf Definition}
\newtheorem{corollary}[theorem]{\bf Corollary}
\newtheorem{proposition}[theorem]{\bf Proposition}

\newtheorem{question}[theorem]{\bf Question}
\newtheorem{conjecture}[theorem]{\bf Conjecture}

\newcommand{\rme}{\mathrm{e}}
\newcommand{\rmi}{\mathrm{i}}

\newcommand{\defeq}{\mathrel{\mathop:}=}

\begin{document}

\title{Twisting and satellite operations on P-fibered braids}

\author{
Benjamin Bode}
\date{}

\address{Department of Mathematics, Osaka University, Toyonaka, Osaka 560-0043, Japan}
\email{ben.bode.2013@my.bristol.ac.uk}




\maketitle
\begin{abstract}
A geometric braid $B$ can be interpreted as a loop in the space of monic complex polynomials with distinct roots. This loop defines a function $g:\mathbb{C}\times S^1\to\mathbb{C}$ that vanishes on $B$. We define the set of \textit{P-fibered} braids as those braids that can be represented by loops of polynomials such that the corresponding function $g$ induces a fibration $\arg g:(\mathbb{C}\times S^1)\backslash B\to S^1$. We show that a certain satellite operation produces new P-fibered braids from known ones. We also prove that any braid $B$ with $n$ strands, $k_-$ negative and $k_+$ positive crossings can be turned into a P-fibered braid (and hence also into a braid whose closure is fibered) by adding at least $\tfrac{k_-+1}{n}$ negative or $\tfrac{k_+ +1}{n}$ positive full twists to it.
\end{abstract}
%

\section{Introduction}\label{sec:intro}

We consider a geometric braid $B$ on $n$ strands as the union of $n$ parametric curves. If the closure of $B$ has $m$ components $\{C_i\}_{i=1,2,\ldots,m}$ made up of $n_i$ strands, we can arrange $B$ to be of the form 
\begin{equation}
\label{eq:braidpara}
\bigcup_{i=1}^{m}\bigcup_{j=1}^{n_i}(z_{i,j}(t),t)\subset\mathbb{C}\times[0,2\pi],\quad t\in[0,2\pi],
\end{equation}
where $z_{i,j}:[0,2\pi]\to\mathbb{C}$ are smooth functions, parametrising the $j$th strand of the component $C_i$, such that for any $t\in[0,2\pi]$ an equality $z_{i,j}(t)=z_{i',j'}(t)$ implies $i=i'$ and $j=j'$ and furthermore, for all $i=1,2,\ldots,m$, and $j=1,2,\ldots,n_i$, there is a unique $k=1,2,\ldots,n_i$, such that $z_{i,j}(2\pi)=z_{i,k}(0)$ and such that the resulting permutation $j\mapsto k$ on $n_i$ points is cyclic. When we refer to a geometric braid $B$ like this, we mean this particular parametrisation, in particular, a choice of labeling the components $C_i$ and the strands $j=1,2,\ldots,n_i$ that is consistent with the conditions on $z_{i,j}$.

Given such a parametrisation we can define the function $g:\mathbb{C}\times[0,2\pi]\to\mathbb{C}$,
\begin{equation}
\label{eq:defg}
g(u,t)=\prod_{i=1}^m\prod_{j=1}^{n_i}(u-z_{i,j}(t)).
\end{equation}
Note that for any value of $t\in[0,2\pi]$, the function $g_t\defeq g(\cdot,t):\mathbb{C}\to\mathbb{C}$ is a complex polynomial in one variable, whose roots are precisely the positions of the $n$ strands of the given braid at the height $t$. Therefore, as $t$ varies from $0$ to $2\pi$, the roots of $g$ trace out the braid $B$ in the parametrisation given in Eq. (\ref{eq:braidpara}). It follows that $g_0=g_{2\pi}$ and we can thus regard a braid as a loop $g_t$, $t\in S^1$ in the space of monic polynomials with distinct roots.


\begin{definition}
A geometric braid $B$ as in Eq. (\ref{eq:braidpara}) is called \textbf{P-fibered} (P for polynomial) if the corresponding function $g$ as in Eq. (\ref{eq:defg}) induces a fibration via 
\begin{equation}
\arg g=g/|g|:(\mathbb{C}\times S^1)\backslash B\to S^1.
\end{equation}

A braid type is called \textit{P-fibered} if it can be represented by a P-fibered geometric braid.
\end{definition}
Note that for every monic complex polynomial $p:\mathbb{C}\to\mathbb{C}$ and every $\chi\in[0,2\pi]$ we have 
\begin{equation}
\label{eq:limit}
\lim_{r\to\infty}\arg p(r\rme^{\rmi \chi})=\rme^{\rmi\chi \deg p }.
\end{equation} 
It follows that we can embed $\mathbb{C}\times S^1$ in $S^3$ such that the fibration extends to $S^3\backslash L$, where $L$ is the closure of the P-fibered braid $B$. Hence closures of P-fibered braids are fibered links in $S^3$. However, it is (to our knowledge) an open problem if every fibered link is the closure of a P-fibered braid.

Although the term `P-fibered braid' has not been used before, the concept has already been applied to the construction of \textit{real algebraic links} in \cite{bode:real, adicact}. These are the real analogue of Milnor's algebraic links \cite{milnor}, that is, they are links of isolated singularities of real polynomial maps from $\mathbb{R}^4$ to $\mathbb{R}^2$.

In this paper, we study satellite operations on P-fibered braids. For a given geometric braid $B$ on $n$ strands, whose closure has components $C_1, C_2,\ldots C_m$, $m\in\mathbb{N}$, we can choose geometric braids $B_1, B_2,\ldots, B_m$ and form the satellite link $L(B;B_1,B_2,\ldots,B_m)$ by replacing a tubular neighbourhood of the component $C_i$ by a solid torus containing the (closed) braid $B_i$. This should be done without any twisting, that is, the images of $\{0\}\times S^1\subset\mathbb{C}\times S^1$, which is identified with $C_i$, and of $\{1\}\times S^1\subset\mathbb{C}\times S^1$ should have a linking number equal to the self-linking number of $C_i$ using the blackboard framing. We denote by $m_i$ the number of components of the closure of $B_i$ and by $n_{i,k}$ the number of strands of the $k$th component of the closure of $B_i$. The braid $B$ has a geometric braid representative given by Eq. (\ref{eq:braidpara}) and $B_i$ has a geometric braid representative parametrised by 
\begin{equation}
\bigcup_{k=1}^{m_i}\bigcup_{\ell=1}^{n_{i,k}}(z_{k,\ell}^i(t),t)\subset\mathbb{C}\times[0,2\pi],\quad t\in[0,2\pi],
\end{equation}
where  $z_{k,\ell}^i:[0,2\pi]\to\mathbb{C}$ is a smooth function, the parametrisation of the $\ell$th strand of the $k$th component of the closure of $B_i$. Since they parametrise a geometric braid, the functions $z_{k,\ell}^i$ satisfy the same conditions as $z_{i,j}$ above. Then the satellite $L(B;B_1,B_2,\ldots,B_m)$ is by definition the closure of the geometric braid 
\begin{equation}
\mathcal{B}(B;B_1,B_2,\ldots,B_m)\defeq\bigcup_{i=1}^m\bigcup_{j=1}^{n_i}\bigcup_{k=1}^{m_i}\bigcup_{\ell=1}^{n_{i,k}}\left(z_{i,j}(t)+\varepsilon z_{k,\ell}^{i}\left(\frac{t+2\pi (j-1)}{n_i}\right),t\right)\subset\mathbb{C}\times[0,2\pi],\quad t\in[0,2\pi],
\end{equation}
with $\varepsilon>0$ chosen sufficiently small. The corresponding loop in the space of monic complex polynomials with distinct roots is thus
\begin{equation}
\label{eq:satbraid}
G_{t,\varepsilon}(u)\defeq G_{\varepsilon}(u,t)\defeq\prod_{i=1}^m\prod_{j=1}^{n_i}\prod_{k=1}^{m_i}\prod_{\ell=1}^{n_{i,k}}\left(u-\left(z_{i,j}(t)+\varepsilon z_{k,\ell}^{i}\left(\frac{t+2\pi (j-1)}{n_i}\right)\right)\right).
\end{equation}

An example of this operation is depicted in Figure \ref{fig:def}. Note that the braid type of $\mathcal{B}(B;B_1,B_2,\ldots,B_m)$ depends on the precise parametrisations of $B$ and $B_i$, $i=1,2,\ldots,m$, and in particular on the labelling $i,j$ of the strands of $B$, while the link type of its closure $L(B;B_1,B_2,\ldots,B_m)$ only depends on the labelling of the components of the closure of $B$.

\begin{figure}
\labellist
\Large
\pinlabel a) at 100 3700
\pinlabel b) at 1500 3700
\pinlabel c) at 3200 3700
\pinlabel d) at 450 2000
\pinlabel e) at 2800 2000
\endlabellist
\centering
\includegraphics[height=10cm]{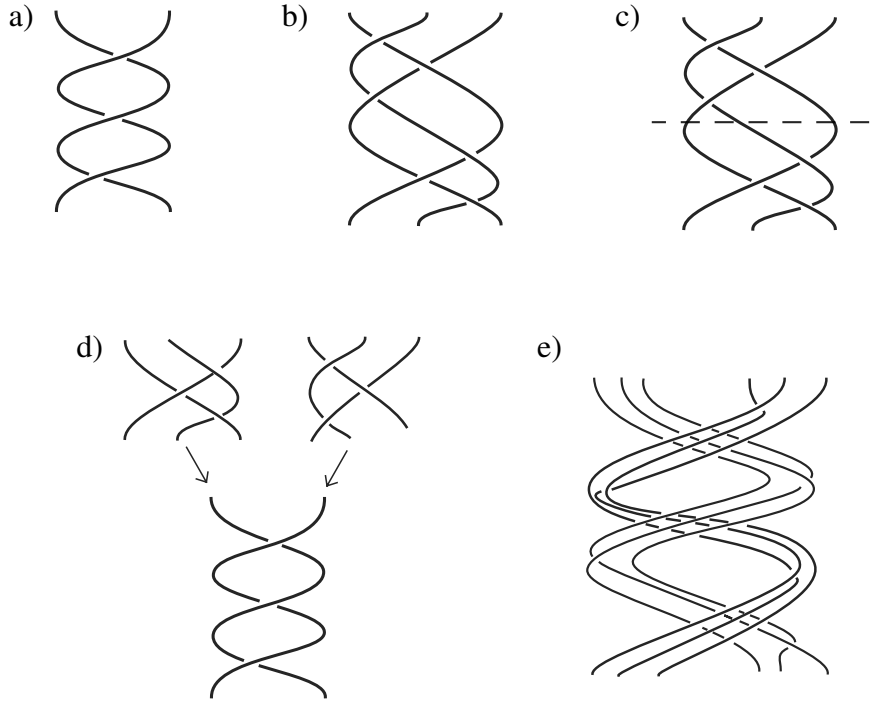}
\caption{An example of the satellite operation. a) The braid $B=\sigma_1^3$ on 2 strands closing to the trefoil knot. b) The braid $B_1=(\sigma_2^{-1}\sigma)^3$, which closes to the Borromean rings. Both of them are P-fibered braids. c) Since the closure of $B$ has one component made up of 2 strands, we divide $B_1$ into two vertical halves and perform the satellite operation as indicated in part d) of the figure. e) The resulting satellite braid $\mathcal{B}(B;B_1)$. \label{fig:def}}
\end{figure}


We obtain a braid diagram from a geometric braid by projecting the strands on the plane $\text{Im}(u)=0$. At a crossing for some value of $t$, we choose the strand with lower $\text{Im}(z_{i,j}(t))$ to be the overpassing strand. Note that this choice of sign is quite arbitrary. If $B$ is a P-fibered braid, then its mirror image is parametrised by the complex conjugates $(\overline{z}_{i,j}(t),t)$ and is hence also P-fibered. From a braid diagram we obtain a braid word in the Artin generators $\sigma_i$, $i=1,2,\ldots,n-1$, and their inverses. We typically read off crossings from a braid diagram from the bottom to the top, i.e., along the direction of increasing $t$.

For a braid $B$ we denote by $B^r$, $r\in\mathbb{N}$, the $r$th power of $B$ in the braid group, that is, the concatenation of $r$ copies of $B$. On the level of braid words, this means that the braid word of $B$ is repeated $r$ times. We write $B^{-r}$ for the $r$th power of $B^{-1}$. Note that if $B$ is a P-fibered braid, then $B^r$ is also P-fibered for all non-zero $r\in\mathbb{Z}$.

The main result of  this paper is as follows.

\begin{theorem}
\label{thm:main}
Let $B$ be a P-fibered braid whose closure has $m$ components. Let $B_1, B_2, \ldots, B_m$ be P-fibered braids with notations as above and such that they all have the same number of strands $s$, i.e., $s=\sum_{k=1}^{m_i} n_{i,k}$ for all $i=1,2,\ldots,m$. Then there are natural numbers $q_i$, $i=1,2,\ldots,m$, such that $\mathcal{B}(B;B_1^{r_1},B_2^{r_2},\ldots,B_m^{r_m})$ is a P-fibered braid for all integers $r_i$ with $|r_i|\geq q_i$.
\end{theorem}

In \cite{mikami} Hirasawa studied under which conditions satelitte knots are fibered. The satellites from Theorem \ref{thm:main} satisfies the conditions in \cite{mikami}, so it is not surprising that the result is a fibered link. The fact that it is a P-fibered braid however is a stronger statement.

Building on a construction from \cite{bode:real} we also show that every link is a sublink of a real algebraic link.
\begin{theorem}
\label{thm:sub}
Every link is a sublink of a real algebraic link.
\end{theorem} 

It follows from the definition of the Garside normal form of braids that every braid $B$ can be made positive by adding a large enough number $k$ of positive full-twists $\Delta^{2k}$. Closures of positive braids are fibered and we use the concept of P-fibered braids to find a new lower bound for $k$ such that $\Delta^{2k}B$ closes to a fibered link for a given braid $B$.
\begin{theorem}
\label{thm:twist}
Let $B$ be a braid on $n$ strands, with $k_-$ negative and $k_+$ positive crossings. Then $\Delta^{2k_1}B$ and $\Delta^{-2k_2}B$ close to fibered links for all $k_1,k_2\in\mathbb{N}$ with $k_1\geq \tfrac{k_-+1}{n}$ and $k_2\geq\tfrac{k_++1}{n}$.
\end{theorem} 

The remainder of this paper is organised as follows. In Section \ref{sec:motiv} we review applications of P-fibered braids to the construction of real algebraic links and simple branched covers, which motivate our results. In Section \ref{sec:operations} we prove Theorem \ref{thm:main}, which is followed by some examples in Section \ref{sec:example}. The proof of Theorem \ref{thm:sub} can be found in Section \ref{sec:sublink}. Theorem \ref{thm:twist} is proved in Section \ref{sec:twist}.


\textbf{Acknowledgements:} This work was supported by JSPS KAKENHI Grant Number JP18F18751 and a JSPS Postdoctoral Fellowship as JSPS International Research Fellow.

\section{Motivation: Real algebraic links and simple branched covers}
\label{sec:motiv}
This chapter reviews applications of P-fibered braids. We omit the proofs and refer to the cited articles, even though the term `P-fibered braid' is not used there.
\subsection{Real algebraic links}
In the last chapter of his seminal work \cite{milnor}, Milnor investigates properties of the so-called \textit{real algebraic links}. We use this term for the real analogue of links of singularities of complex hypersurfaces, namely links of isolated critical points of polynomials $f:\mathbb{R}^4\to\mathbb{R}^2$. This should not be confused with knotted algebraic varieties in $\mathbb{RP}^3$ as they were introduced by Viro \cite{viro}, which are also referred to as real algebraic links.

\begin{definition}
\label{def:iso}
Let $p:\mathbb{R}^4\to\mathbb{R}^2$ be a polynomial map. We say $p$ has an \textbf{isolated singularity} at the origin if $p(0)=0$, the Jacobian matrix satisfies $\nabla p(0)=0$ and there is a neighbourhood $B$ of $0\in\mathbb{R}^4$ such that $0$ is the only point in $B$ where the rank of $\nabla p$ is not full.
\end{definition}

The number $0$ in Definition \ref{def:iso} refers to the origin in $\mathbb{R}^4$ and $\mathbb{R}^2$ and the zero matrix of size 2-by-4, respectively. Like in the complex case, an isolated singularity at the origin guarantees that the intersection of the vanishing set
\begin{equation}
V_p\defeq p^{-1}(0)=\{(x_1,x_2,x_3,x_4)\in\mathbb{R}^4:p(x_1,x_2,x_3,x_4)=(0,0)\}
\end{equation}
and the 3-sphere $S^3_{\rho}$ of small enough radius $\rho>0$ is a link, whose link type is independent of the radius.

\begin{definition}
\label{def:ral}
A link $L$ is \textbf{real algebraic} if there exists a polynomial $p:\mathbb{R}^4\to\mathbb{R}^2$ such that $p$ has an isolated singularity at the origin and $p^{-1}(0)\cap S_{\rho}^3$ is isotopic to $L$ for all small enough radii $\rho$. We also say $L$ is the \textbf{link of the singularity} of $p$.
\end{definition}

Milnor showed that all real algebraic links are fibered \cite{milnor}, but it is not known which links are real algebraic. 
\begin{conjecture}[Benedetti-Shiota \cite{benedetti}]
\label{conj:bene}
A link is real algebraic if and only if it is fibered.
\end{conjecture} 
So far, the set of links that are known to be real algebraic is still comparatively small. There are of course the algebraic links (coming from complex polynomials), but also certain unions of algebraic links \cite{pichon} and the connected sum $K\# K$ of any fibered knot $K$ with itself \cite{looijenga}. Perron \cite{perron} and Rudolph \cite{rudolph:isolated} constructed polynomials for the figure-eight knot and it was shown in \cite{bode:real} that all closures of squares of homogeneous braids are real algebraic (cf. Definition \ref{def:homo}).

The concept of P-fibered braids had not been introduced at the time of \cite{bode:real}, but now we can rephrase the results as follows.
\begin{definition}
\label{def:homo}
A braid $B$ on $n$ strands is called \textit{homogeneous} if it can be written as a word $w$ in the Artin generators such that for every $i=1,2,\ldots,n-1$ the generator $\sigma_i$ appears in $w$ if and only if $\sigma_i^{-1}$ does not appear.
\end{definition}
\begin{theorem}\label{thm:homo}\cite{bode:real}
Let $B$ be a homogeneous braid. Then $B$ is P-fibered.
\end{theorem}
\begin{theorem}\label{thm:square}\cite{bode:real}
Let $B$ be a P-fibered braid. Then the closure of $B^2$ is real algebraic.
\end{theorem}
A more recent and more detailed account of these results can also be found in \cite{bode:review}. The fact that homogeneous braid closures are fibered is due to Stallings \cite{stallings}.

A construction that was shown to go beyond the family of homogeneous braids has been proposed in \cite{adicact}. Both constructions in \cite{bode:real} and \cite{adicact} allow us to write down the real polynomial map $p:\mathbb{R}^4\to\mathbb{R}^2$ explicitly in terms of $g_t$. In particular, $p$ can be written as a \textit{mixed polynomial} in complex variables $u$, $v$ and the conjugate $\overline{v}$, i.e., it is holomorphic with respect to one of the complex variables, a property that we call \textit{semiholomorphic}. The degree of $p$ with respect to $u$ is equal to the degree of $g_t$ and hence the number of strands of the P-fibered braid.

\subsection{Simple branched covers and symmetric Hopf plumbings}
In this section we will explain how closures of P-fibered braids are also examples of fibered links that can be constructed from the unknot via a sequence of Hopf plumbings and deplumbings that are in some sense symmetric with respect to certain simple branched covers.

Harer conjectured that every fibered link could be obtained from the unknot by a sequence of Hopf plumbings and deplumbings. This was eventually proven by Giroux \cite{giroux1, giroux} by establishing the correspondence between contact structures and open book decompositions.

Before Giroux's proof, Montesinos and Morton \cite{morton} attempted to prove Harer's conjecture with the following approach.

Given a link $L$ in $S^3$ we say an unknot $\alpha\subset S^3\backslash L$ is a \textit{braid axis} of $L$ if the fibration of the unknot complement $S^3\backslash \alpha$ can be arranged such that each fiber intersects $L$ transversally in the same number of points.

\begin{definition}
A $d$-sheeted branched covering map $\pi:F\to S$ between two surfaces $F$ and $S$ is called simple if for every point $p$ in the finite branch set $Q\subset S$ the preimage $\pi^{-1}(p)$ consists of $d-1$ points. 

A map $\pi:M\to N$ between closed 3-manifolds $M$ and $N$ is a \textbf{simple branched cover} with branch set $C\subset N$ if it is locally homeomorphic to the product of an interval with a simple $d$-sheeted branched cover of a disc, and the branch points in the products form the set $C$. 

\end{definition}

Let $\pi:S^3\to S^3$ be a simple branched cover, branched over a link $L$. Then for every braid axis $\alpha$ of $L$ the link $\pi^{-1}(\alpha)$ is fibered. Hilden and Montesinos showed that the converse is also true \cite{hilden}. For every fibered link $L'$ there exists such a simple branched cover $\pi$, branched over some link $L$ and a braid axis $\alpha$ of $L$, such that $L'=\pi^{-1}(\alpha)$.

Montesinos and Morton observed that for a fixed simple branched cover $\pi:S^3\to S^3$ with branch set $L$ the preimages $\pi^{-1}(\alpha)$ and $\pi^{-1}(\beta)$ of two braid axes $\alpha$ and $\beta$ of $L$ are related by a sequence of Hopf plumbings and deplumbings \cite{morton}, which led to the following question:

\begin{question}
\label{q:morton}
Is there for every fibered link $L'$ a simple branched cover $\pi:S^3\to S^3$, branched over a link $L$ with braid axes $\alpha$ and $\beta$, such that $\pi^{-1}(\alpha)=L'$ and $\pi^{-1}(\beta)$ is an unknot?
\end{question}

It is (to our knowledge) not known if every Hopf plumbing can be described in this way (i.e., a change of the braid axis for the branch link $L$). Hence, while a positive answer to this question implies another proof of Harer's conjecture, Giroux's work does not imply an answer to this question.

In order for a Hopf (de)plumbing operation on a fiber surface $F$ to arise in this way, the attaching arc $\gamma\subset F$ has to be somewhat symmetric with respect to the simple branched covering map $\pi$. Namely, there has to be a path $\gamma'$ in the disk $S$ from $L\cap S$ to $\partial S$, disjoint from $L$ apart from its starting point, such that $\pi(\gamma)=\gamma'$. If a Hopf plumbing or deplumbing occurs along an arc $\gamma_i$, $i=1,2,\ldots,N$, like this, we say the Hopf (de)plumbing is \textit{symmetric}. 

Let $D$ be the open unit disk in $\mathbb{C}$ with closure $\overline{D}$. Note that $(\overline{D}\times S^1)/((\rme^{\rmi \chi},\rme^{\rmi t_1})\sim(\rme^{\rmi \chi},\rme^{\rmi t_2}))\cong S^3$. Let $\phi: \mathbb{C}\to D$ be an orientation-preserving diffeomorphism such as $\phi(u)=\frac{u}{1+|u|}$. Given a P-fibered braid $B$ on $n$ strands and the corresponding function $g$, we can define the map $\pi:S^3\to S^3$,
\begin{align}
\pi(\phi(u),\rme^{\rmi t})&=(\phi^{-1}(g(u,t)),\rme^{\rmi t}),\nonumber\\
\pi(\rme^{\rmi \chi},\rme^{\rmi t})&=(\rme^{\rmi \chi n},\rme^{\rmi t}).
\end{align}

Using Eq. (\ref{eq:limit}) and basic properties of complex polynomials it is not difficult to check that $\pi$ is a simple branched cover, branched over the link that is the closure of the critical values $(v_1(t),v_2(t),\ldots,v_{n-1}(t))$ of $g_t$. One braid axis of this link is given by $\beta=(\rme^{\rmi s},\rme^{\rmi t})$, with $s$ varying from 0 to $2\pi$, whose preimage $\pi^{-1}(\beta)$ is the unknot $(\rme^{\rmi \chi},\rme^{\rmi t})$, with $\chi$ varying from 0 to $2\pi$. Another braid axis is $\alpha=(0,\rme^{\rmi t})$, with $t$ varying from 0 to $2\pi$, whose preimage $\pi^{-1}(\alpha)$ is the closure of $B$.

\begin{proposition}(cf. \cite{survey})
\label{prop:morton}
Let $B$ be a P-fibered braid. Then $\pi:S^3\to S^3$ as above is a simple branched cover, branched over a link $L$ with braid axes $\alpha$ and $\beta$, such that $\pi^{-1}(\alpha)$ is the closure of $B$ and $\pi^{-1}(\beta)$ is an unknot. 
\end{proposition}
\begin{corollary}
Let $B$ be a P-fibered braid. Then its closure can be obtained from the unknot by a sequence of symmetric Hopf plumbings and deplumbings.
\end{corollary}

\begin{corollary}
Let $B$ be a homogeneous braid. Then there is a simple branched cover as in Proposition \ref{prop:morton}. Hence the closure of $B$ can be obtained from the unknot by a sequence of symmetric Hopf plumbings and deplumbings.
\end{corollary}

The constructions in \cite{bode:real} and \cite{adicact} as well as Theorem \ref{thm:main} can thus be motivated by making progress with regards to Conjecture \ref{conj:bene} and Question \ref{q:morton}. P-fibered braids (together with their braid axes) are also examples of \textit{generalized exchangeable braids} as in \cite{rampi} and \textit{mutually braided open books} as in \cite{mutual}.

\section{The proof of Theorem \ref{thm:main}}
\label{sec:operations}
\subsection{The proof}
It is a simple calculation that $\arg g$ is a fibration if and only if for all $(u_*,t_*)\in(\mathbb{C}\times S^1)\backslash B$:
\begin{equation}
\label{eq:fib}
\left(\frac{\partial g}{\partial u}(u_*,t_*),\frac{\partial \arg g}{\partial t}(u_*,t_*)\right)\neq (0,0).
\end{equation}

If we denote the critical points of $g_t$, i.e., the roots of the derivative $\tfrac{\partial g_t}{\partial u}$, by $c_p(t)$, $p=1,2,\ldots,n-1$, and set the critical values to be $v_p(t)\defeq g_t(c_p(t))$, $p=1,2,\ldots,n-1$, then Eq. (\ref{eq:fib}) becomes
\begin{equation}
\frac{\partial \arg v_p(t)}{\partial t}\neq 0,\qquad \text{for all }p=1,2,\ldots,n-1,\quad t\in[0,2\pi].
\end{equation}
Since the condition in Eq. (eq:fib) is an open condition, i.e., it remains true under small perturbations of the coefficients of $g_t$, we can assume without loss of generality that the critical points $c_p(t)$, $p=1,2,\ldots,n-1$, are distinct for any $t\in[0,2\pi]$.

We denote the critical points of $G_{t,\varepsilon}$ in Eq. (\ref{eq:satbraid}) by $C_{\varepsilon,p}(t)$, $p=1,2,\ldots,ns-1$, and its critical values by $V_{\varepsilon,p}(t)\defeq G_{t,\varepsilon}(C_{\varepsilon,p}(t))$. In order to prove Theorem \ref{thm:main}, we have to show that 
\begin{equation}
\frac{\partial \arg V_{\varepsilon,p}(t)}{\partial t}\neq 0,\qquad \text{for all }p=1,2,\ldots,ns-1,\quad t\in[0,2\pi]
\end{equation}
if $\varepsilon>0$ is chosen sufficiently small.




\begin{proof}[Proof of Theorem \ref{thm:main}]
Say the $j$th strand of the component $C_i$ of the closure of $B$ is parametrised by $(z_{i,j}(t),t)\subset\mathbb{C}\times[0,2\pi]$ such that $\arg g:\mathbb{C}\times[0,2\pi]\backslash B\to S^1$ is a fibration. We consider
\begin{equation}
g_t^s(u)\defeq g^s(u,t)\defeq\prod_{i,j}(u-z_{i,j}(t))^s.
\end{equation}
Since $c_p(t)$, $p=1,2,\ldots,n-1$, is a critical point of $g_t$, it is also a critical point of $g_t^s$ (with multiplicity one). The other critical points of $g_t^s$ are the roots $z_{i,j}(t)$, $i=1,2,\ldots,m$, $j=1,2,\ldots,n_i$, of $g_t$ and they have multiplicity $s-1$. We call a critical point of multiplicity one \textit{simple}. The critical points of $g_t^s$ can therefore be divided into two groups: Those critical points that are also critical points of $g_t$ (and are hence simple) and those that are not. 

Since $\arg g$ is a fibration, the critical values $v_p(t)=g(c_p(t),t)$ of $g$ have the property that $\tfrac{\partial \arg v_p(t)}{\partial t}$ never vanishes. It follows from $\arg g^s(c_p,t)=s\arg g(c_p,t)$ that the $n-1$ critical values $v_p(t)^s=g^s(c_p(t),t)$ of $g^s$ have the same property if $c_p(t)$ is a simple critical point, i.e., $\tfrac{\partial \arg v_p(t)^s}{\partial t}$ never vanishes.

Note that $G_{t,0}=g_t^s$. Hence, for any critical point of $g^s_t$ of multiplicity $m$ there is a neighbourhood $U\subset\mathbb{C}$ such that for all $\varepsilon>0$ chosen sufficiently small $G_{t,\varepsilon}$ has $m$ critical points in $U$ when counted with multiplicity. 

The values that a polynomial takes on its simple critical points are smooth functions of the coefficients. We can thus think of $n-1$ of the critical values $V_{\varepsilon,p}(t)$, say $p=1,2,\ldots,n-1$, of $G_{t,\varepsilon}$ as smooth functions $V_p(t)(\varepsilon)$ of $\varepsilon\geq 0$ with $V_p(t)(0)=v_p(t)^s$, where $v_p(t)=c_p(t)$ for a simple critical point $c_p(t)$ of $g_t$. Since $\arg g_t^s$ is a fibration, $\tfrac{\partial \arg v_p(t)^s}{\partial t}$ never vanishes and hence $\tfrac{\partial \arg V_p(t)(\varepsilon)}{\partial t}$ never vanishes as long as $\varepsilon$ is chosen sufficiently small.

What remains is a proof that the same statement holds for the critical values of $G_{t,\varepsilon}$ that do not correspond to simple critical points of $g_t^s$. These are the values $V_{\varepsilon,p}(t)$, $p=n,n+1,\ldots,ns-1$, that $G_{t,\varepsilon}$ takes on the critical points $C_{\varepsilon,p}(t)$ that lie in arbitrarily small neighbourhoods of $z_{i,j}(t)$, i.e., the roots of $g_t$ and therefore the critical points of $g_t^s$ of multiplicity $s-1$, when $\varepsilon$ is chosen small enough.

Let 
\begin{equation}
g_{i,j}(u,t)\defeq\prod_{k=1}^{m_i}\prod_{\ell=1}^{n_{i,k}}\left(u-z_{k,\ell}^i\left(\frac{t+2\pi (j-1)}{n_i}\right)\right)
\end{equation}
be the polynomial map whose vanishing set is the $j$th interval $[\tfrac{2\pi (j-1)}{n_i},\tfrac{2\pi j}{n_i}]$ of the braid $B_i$, i.e., $B_i\cap (\mathbb{C}\times [\tfrac{2\pi (j-1)}{n_i},\tfrac{2\pi j}{n_i}])$. Note that in general $g_{i,j}(u,2\pi)\neq g_{i,j}(u,0)$, but rather $g_{i,j}(u,2\pi)=g_{i,j+1}(u,0)$. We denote the critical points of $g_{i,j}(\cdot,t):\mathbb{C}\to\mathbb{C}$ by $c_{i,j,p}(t)$, $p=1,2,\ldots,s-1$. Since $B_i$ is assumed to be a P-fibered braid with parametrisation $(z_{k,\ell}^i(t),t)$, the critical values $v_{i,j,q}(t)\defeq g_{i,j}(c_{i,j,q}(t),t)$ of $g_{i,j}$ satisfy $\tfrac{\partial \arg v_{i,j,q}(t)}{\partial t}\neq 0$ for all $q=1,2,\ldots,s-1$, $t\in[0,2\pi]$. 

\ \\

\textbf{Claim:} For $\varepsilon>0$ sufficiently small, say less than some $\delta>0$, the critical points $C_{\varepsilon,p}(t)$, $p=n,n+1,\ldots,ns-1$, of $G_{t,\varepsilon}$, which do not correspond to simple critical points of $g_t^s$, are given by $z_{i,j}(t)+\varepsilon \varphi_{i,j,q}(t,\varepsilon)$ for some $i=1,2,\ldots,m$, $j=1,2,\ldots,n_i$ and $\varphi_{i,j,q}:[0,2\pi]\times[0,\delta]\to\mathbb{C}$, $q=1,2,\ldots,s-1$, a smooth function with $\varphi_{i,j,q}(t,0)=c_{i,j,q}(t)$.

Fix a choice of $i$ and $j$. Since $\arg g_{i',j'}$, $i'=1,2,\ldots,m$, $j=1,2,\ldots,n_{i'}$, being a fibration is an open condition, i.e., Eq. (\ref{eq:fib}) remains true under small perturbations of the coefficients of $g_{i',j'}$, we can assume without loss of generality  that $z_{k,\ell}^{i'}(t)\neq 0$ for all $i'=1,2,\ldots,m$, $k=1,2,\ldots,m_{i'}$, $\ell=1,2,\ldots,n_{i',k}$, and all $t\in[0,2\pi]$. To prove the claim we consider $G_{t,\varepsilon}$ as a continuous family of rational maps from the Riemann sphere $\mathbb{C}\cup\{\infty\}$ to itself. We can apply a change of coordinates $u\mapsto \widetilde{u}-z_{i,j}$ to $G_{t,\varepsilon}$ such that the factors in Eq. (\ref{eq:satbraid}) that correspond to $i$ and $j$ are of the form $u-\varepsilon z_{k,\ell}^i$. We divide the whole expression by $\varepsilon^{ns}$ and apply another coordinate change $\widetilde{u}\mapsto \widehat{u}\varepsilon$ such that for all $\varepsilon>0$ the function $G_{t,\varepsilon}(\widehat{u},t)$ has the same critical points as
\begin{align}
\frac{1}{\varepsilon^{ns}}G_{t,\varepsilon}(\widehat{u},t)=&\prod_{k=1}^{m_i}\prod_{\ell=1}^{n_{i,k}}\left(\widehat{u}-z_{k,\ell}^i\left(\frac{t+2\pi (j-1)}{n_i}\right)\right)\nonumber\\
&\times \underset{(i',j')\neq (i,j)}{\prod_{(i',j')=(1,1)}^{(m,n_i)}}\prod_{k=1}^{m_i'}\prod_{\ell=1}^{n_{i',k}}\left(\widehat{u}-\frac{1}{\varepsilon}\left(z_{i',j'}(t)+\varepsilon z_{k,\ell}^{i'}\left(\frac{t+2\pi (j'-1)}{n_{i'}}\right)\right)\right)\nonumber\\
=&g_{i,j}(\widehat{u},t)\underset{(i',j')\neq (i,j)}{\prod_{(i',j')=(1,1)}^{(m,n_i)}}\prod_{k=1}^{m_i'}\prod_{\ell=1}^{n_{i',k}}\left(\widehat{u}-\frac{1}{\varepsilon}\left(z_{i',j'}(t)+\varepsilon z_{k,\ell}^{i'}\left(\frac{t+2\pi (j'-1)}{n_{i'}}\right)\right)\right).
\end{align}

As $\varepsilon$ goes to zero, the roots of $G_{t,\varepsilon}(\widehat{u},t)$ that are not roots of $g_{i,j}(\widehat{u},t)$ tend to infinity, canceling with the pole at infinity of multiplicity $ns$, such that the limit function is the first factor $g_{i,j}(\widehat{u},t)$. This can be made explicit by another change of variable, where we consider $w=\tfrac{1}{\widehat{u}}$, leading to
\begin{align}
\label{eq:limit}
\frac{1}{\varepsilon^{ns}}G_{t,\varepsilon}(w,t)&=g_{i,j}(w,t)\underset{(i',j')\neq (i,j)}{\prod_{(i',j')=(1,1)}^{(m,n_i)}}\prod_{k=1}^{m_i'}\prod_{\ell=1}^{n_{i',k}}\left(\frac{\left(z_{i',j'}(t)+\varepsilon z_{k,\ell}^{i'}\left(\frac{t+2\pi (j'-1)}{n_{i'}}\right)\right)}{\varepsilon w}\right.\nonumber\\
&\hspace{4cm}\times\left.\left(\varepsilon\left(z_{i',j'}(t)+\varepsilon z_{k,\ell}^{i'}\left(\frac{t+2\pi (j'-1)}{n_{i'}}\right)\right)^{-1}-w\right)\right)\nonumber\\
&=g_{i,j}(w,t)\frac{1}{\varepsilon^{(n-n_i)s}}\underset{(i',j')\neq (i,j)}{\prod_{(i',j')=(1,1)}^{(m,n_i)}}\prod_{k=1}^{m_i'}\prod_{\ell=1}^{n_{i',k}}\left(\frac{\left(z_{i',j'}(t)+\varepsilon z_{k,\ell}^{i'}\left(\frac{t+2\pi (j'-1)}{n_{i'}}\right)\right)}{w}\right.\nonumber\\
&\hspace{5cm}\times\left.\left(\varepsilon\left(z_{i',j'}(t)+\varepsilon z_{k,\ell}^{i'}\left(\frac{t+2\pi (j'-1)}{n_{i'}}\right)\right)^{-1}-w\right)\right)
\end{align}

We can multiply the whole expression by $\varepsilon^{(n-n_i)s}$ and find that the right hand side has a well-defined limit as $\varepsilon$ goes to zero, which is
\begin{equation}
g_{i,j}(w,t)\underset{(i',j')\neq (i,j)}{\prod_{(i',j')=(1,1)}^{(m,n_i)}}\prod_{k=1}^{m_i'}\prod_{\ell=1}^{n_{i',k}}\frac{z_{i',j'}(t)}{w}\left(-w\right)=g_{i,j}(w,t)\underset{(i',j')\neq (i,j)}{\prod_{(i',j')=(1,1)}^{(m,n_i)}}\prod_{k=1}^{m_i'}\prod_{\ell=1}^{n_{i',k}}-z_{i',j'}(t),
\end{equation}
which is simply a non-zero multiple of $g_{i,j}(w,t)$ and therefore has the same critical points as $g_{i,j}(w,t)$.

Therefore the critical points of $G_{t,\varepsilon}(w,t)$ are identical to the critical points of the right hand side of Eq. (\ref{eq:limit}) for all $\varepsilon>0$ and for $\varepsilon=0$ the critical points of the right hand side of Eq. (\ref{eq:limit}) are equal to the critical points of $g_{i,j}(w,t)$.

Since simple critical points of rational maps are smooth functions of their coefficients and the critical points of $g_{i,j}$ are simple, the critical points of $G_{t,\varepsilon}$ are simple as long as $\varepsilon$ is sufficiently small, say less than some $\delta>0$, and they are given by smooth functions of $\varepsilon\in(0,\delta)$. Considering the variable changes that we have used, we find that they take the claimed form of $z_{i,j}(t)+\varepsilon \varphi_{i,j,q}(t,\varepsilon)$. Since $\varphi_{i,j,q}(t,\varepsilon)$ is smooth on $[0,2\pi]\times(0,\delta)$ it has a well-defined limit function as $\varepsilon$ goes to zero, which is precisely $c_{i,j,q}(t)$. 

Strictly speaking, we have used these arguments for a fixed value of $t$, but since $t$ takes values in the compact set $S^1$, we can guarantee that for all $t$ the critical points $C_{\varepsilon,p}(t)$, $p=n,n+1,\ldots,ns-1$, of $G_{t,\varepsilon}$, which do not correspond to simple critical points of $g_t^s$, take the form $z_{i,j}+\varepsilon \varphi_{i,j,q}(t,\varepsilon)$ with $\varphi_{i,j,q}(t,0)=c_{i,j,q}(t)$ and $q\in\{1,2,\ldots,s-1\}$.

This proves the claim. We will now see how it implies that the satellite braid is P-fibered. Let $C_{\varepsilon,p}(t)=z_{i,j}+\varepsilon \varphi_{i,j,q}(t,\varepsilon)$ be a critical point of $G_{t,\varepsilon}$ that does not correspond to a simple critical point of $g_t^s$ and let $V_{\varepsilon,p}(t)=G_{t,\varepsilon}(C_{\varepsilon,p}(t))$.

We calculate
\begin{equation}
\frac{\partial \arg V_{\varepsilon,p}(t)}{\partial t}=\frac{\partial }{\partial t}\text{Im Log\ }G_{t,\varepsilon}(z_{i,j}+\varepsilon \varphi_{i,j,q}(t,\varepsilon))=T_1+T_2,
\end{equation}
where

\begin{align}
\label{eq:t1}
T_1&\defeq\sum_{k=1}^{m_i}\sum_{\ell=1}^{n_{i,k}}\left\{\frac{1}{\left|\varepsilon \left(\varphi_{i,j,q}(t,\varepsilon)-z_{k,\ell}^i\left(\frac{t+2\pi (j-1)}{n_i}\right)\right)\right|^2}\right.\nonumber\\
&\hspace{1cm}\left.\times\left[\text{Re}\left(\varepsilon\left(\varphi_{i,j,q}(t,\varepsilon)-z_{k,\ell}^i\left(\frac{t+2\pi (j-1)}{n_i}\right)\right)\right)\text{Im}\left(\varepsilon\tfrac{\partial \left(\varphi_{i,j,q}(t,\varepsilon)-z_{k,\ell}^i\left(\frac{t+2\pi (j-1)}{n_i}\right)\right)}{\partial t}\right)\right.\right.\nonumber\\
&\hspace{1cm}\ \ \ -\left.\left.\text{Im}\left(\varepsilon\left(\varphi_{i,j,q}(t,\varepsilon)-z_{k,\ell}^i\left(\frac{t+2\pi (j-1)}{n_i}\right)\right)\right)\text{Re}\left(\varepsilon\tfrac{\partial \left(\varphi_{i,j,q}(t,\varepsilon)-z_{k,\ell}^i\left(\frac{t+2\pi (j-1)}{n_i}\right)\right)}{\partial t}\right)\right]\right\}
\end{align}
and
\begin{align}
\label{eq:t2}
T_2&\defeq\underset{(i',j')\neq (i,j)}{\sum_{(i',j')=(1,1)}^{m,n_{i'}}}\left\{\frac{1}{\left|z_{i,j}(t)+\varepsilon \varphi_{i,j,q}(t,\varepsilon)-z_{i',j'}(t)-\varepsilon z_{k,\ell}^{i'}\left(\frac{t+2\pi (j'-1)}{n_{i'}}\right)\right|^2}\right.\nonumber\\
&\hspace{1.5cm}\left.\times\left[\text{Re}\left(z_{i,j}(t)+\varepsilon\varphi_{i,j,q}(t,\varepsilon)-z_{i',j'}(t)-\varepsilon z_{k,\ell}^{i'}\left(\frac{t+2\pi (j'-1)}{n_{i'}}\right)\right)\right.\right.\nonumber\\
&\hspace{2cm}\left.\left.\times\text{Im}\left(\tfrac{\partial \left(z_{i,j}(t)+\varepsilon\varphi_{i,j,q}(t,\varepsilon)-z_{i',j'}(t)-\varepsilon z_{k,\ell}^{i'}\left(\frac{t+2\pi (j'-1)}{n_i}\right)\right)}{\partial t}\right)\right.\right.\nonumber\\
&\hspace{1.7cm}-\left.\left.\text{Im}\left(z_{i,j}(t)+\varepsilon\varphi_{i,j,q}(t,\varepsilon)-z_{i',j'}(t)-\varepsilon z_{k,\ell}^{i'}\left(\frac{t+2\pi (j'-1)}{n_{i'}}\right)\right)\right.\right.\nonumber\\
&\hspace{2cm}\left.\left.\times\text{Re}\left(\tfrac{\partial \left(z_{i,j}(t)+\varepsilon\varphi_{i,j,q}(t,\varepsilon)-z_{i',j'}(t)-\varepsilon z_{k,\ell}^{i'}\left(\frac{t+2\pi (j'-1)}{n_{i'}}\right)\right)}{\partial t}\right)\right]\right\}.
\end{align}

Note that the claim above imlies that the first term $T_1$ becomes $\tfrac{\partial \arg g_{i,j}(c_{i,j,q}(t),t)}{\partial t}$ as $\varepsilon$ goes to zero. Since $B_i$ is P-fibered, this expression is non-zero. The second term also has a well-defined limit, which only depends on differences of the roots of $g_t$, i.e., on $z_{i,j}(t)$ and $z_{i',j'}(t)$. 

If we perform this same calculation not for $B_i$, but for $B_i^{r_i}$, we have to multiply the first term by $r_i$, while the limit of the second term is unchanged. It follows that as long as $r_i$ has a sufficiently large modulus, the whole expression becomes non-zero. Again this can be achieved for all $i,j,p$ and $t$ simultaneously because of compactness of $S^1$.

This concludes the proof of Theorem \ref{thm:main}. 


\end{proof}

\subsection{Some examples}
\label{sec:example}
We conclude this section with some examples.

\textbf{Example 1:}
Consider the Hopf link as the closure of the P-fibered braid given by
\begin{align}
z_{1}(t)&=\rme^{\rmi t}, &t\in[0,2\pi],\\
z_{2}(t)&=-\rme^{\rmi t}, &t\in[0,2\pi].
\end{align}


Since both components $C_1$ and $C_2$ of the closure consist of only one strand we have omitted the index $j$, which usually runs through the strands making up one component, from the parametrisation.

We want to perform the satellite operation with $B_1=(\sigma_2^{-1}\sigma_1^{-1})^2$, a P-fibered braid closing to the (negative) trefoil knot, and $B_2=(\sigma_1\sigma_2^{-1})^2$, a P-fibered braid closing to the figure eight knot. Both are braids on three strands, so that the condition from Theorem \ref{thm:main} is satisfied.

Since both braid closures have only one component, we omit the index $k$ (which usually runs through the components) in their parametrisations. $B_1$ and $B_2$ have the following parametrisations as P-fibered braids \cite{bode:2016lemniscate}:
\begin{align}
z_{\ell}^1(t)&=\rme^{\rmi (-2t+2\pi \ell)/3}, &\ell=1,2,3,\ t\in[0,2\pi]\\
z_{\ell}^2(t)&=\cos\left(\frac{2t+2\pi \ell}{3}\right)+\frac{\rmi}{2}\sin\left(\frac{2(2t+2\pi \ell)}{3}\right), &\ell=1,2,3,\ t\in[0,2\pi].
\end{align}

By the definition of the satellite operation we get a braid that is parametrised by
\begin{equation}
\bigcup_{i=1}^2\bigcup_{\ell=1}^3\left(z_{i,j}(t)+\varepsilon z_{\ell}^i(t),t\right)\subset\mathbb{C}\times[0,2\pi].
\end{equation}

We define $G_{t,\varepsilon}$ as in Eq. (\ref{eq:satbraid}) and find that for $\varepsilon=1/10$ we obtain a fibration $\arg G_{t,1/10}:(\mathbb{C}\times S^1)\to S^1$. 



The braid word of the satellite braid is given by 
\begin{equation}
B(B;B_1,B_2)=\sigma_1\sigma_2^{-1}\sigma_5^{-1}\sigma_4^{-1}\sigma_3\sigma_4\sigma_5\sigma_2\sigma_3\sigma_4\sigma_1\sigma_2\sigma_3\sigma_2^{-1}\sigma_1^{-1}\sigma_4\sigma_5^{-1}\sigma_3\sigma_4\sigma_5\sigma_2\sigma_3\sigma_4\sigma_1\sigma_2\sigma_3.
\end{equation}

\ \\

\textbf{Example 2:} We should look at one more example where the closure of $B$ has components that consist of more than one strand. For this we take the figure-eight braid $B=(\sigma_1\sigma_2^{-1})^2$ from before, so
\begin{equation}
z_{j}(t)=\cos\left(\frac{2t+2\pi j}{3}\right)+\frac{\rmi}{2}\sin\left(\frac{2(2t+2\pi j)}{3}\right), \qquad j=1,2,3.
\end{equation}
We have omitted the index $i$, since the closure only has one component.

For the satellite operation we choose the 2-strand braid $B_1=\sigma_1$. Since the figure-eight knot only has one component, we can only pick one braid and the condition on the number of strands in Theorem \ref{thm:main} becomes obsolete.

The braid $B_1$ is a P-fibered braid given by
\begin{equation}
z_{\ell}(t)=\rme^{\rmi (t+\pi \ell)/2},\qquad \ell=1,2.
\end{equation}
Again we have omitted the index $k$ for the same reason as above.

We obtain the satellite braid as
\begin{equation}
\label{eq:exx}
\bigcup_{j=1}^3\bigcup_{\ell=1}^3\left(z_{j}(t)+\varepsilon z_{\ell}\left(\frac{t+2\pi j}{3}\right),t\right)\subset\mathbb{C}\times[0,2\pi].
\end{equation}

Using Mathematica we find that (depending on the strand of $B$) $T_2$ in Eq. (\ref{eq:t2}) takes values between -2.77128 and 2.77128 (rounded) when $t$ goes from 0 to $2\pi$, while $T_1$ in Eq. (\ref{eq:t1}) is constant at $1/3$ as $\varepsilon$ goes to zero. It follows that $\tfrac{\partial \arg V_{\varepsilon,p}(t)}{\partial t}=T_1+T_2$ has zeros for some $p\in\{1,2,3,4,5\}$ even as $\varepsilon$ goes to zero and hence the geometric braid in Eq. (\ref{eq:exx}) is not P-fibered when $\varepsilon$ is small. However if we use $B_1^{r_1}$ with $|r_1|\geq 9\geq 3\times2.77128$ instead of $B_1$, we obtain a P-fibered braid $\mathcal{B}(B;B_1^{r_1})$.

\section{Every link is a sublink of a real algebraic link}
\label{sec:sublink}
We say that a link $L$ is a \textit{sublink} of a link $L'$ if $L'$ is the disjoint union of $L$ and some other link $L''$. In other words, deleting the components of $L'$ that make up $L''$ leaves us with $L$.

It is known that every link is a sublink of a fibered link. One proof of this fact by Stallings is by taking a braid $B$ that closes to the link in question and add one more strand to it, which threads through $B$ in such a way that it creates an alternating and hence homogeneous braid \cite{stallings}. We are going to use a very similar technique to show that every link is a sublink of a real algebraic link using Theorem \ref{thm:homo} and Theorem \ref{thm:square} 

\begin{proof}[Proof of Theorem \ref{thm:sub}]
Let $B$ be a braid on $n$ strands. We obtain another braid $B'$ on $2n$ strands by replacing each generator in the word of $B$ as follows:
\begin{align}
\sigma_i^{(-1)^{i+1}}&\mapsto\sigma_i^{(-1)^{i+1}},\nonumber\\
\sigma_i^{(-1)^{i}}&\mapsto \left(\prod_{j=1}^{n-i-1}\sigma_{n+1-j}^{(-1)^{n-j}}\right)\sigma_{i+1}^{(-1)^{i}}\prod_{j=1}^s\sigma_{i}^{(-1)^{i+1}}.
\end{align}

The rule is depicted in Figure \ref{fig:subhomo}. Geometrically, $B'$ is obtained from $B$ by adding $n$ strands to the right of it and threading the first of the added strands through $B$ in an alternating fashion such that it shifts the crossing $\sigma_{i}^{(-1)^{i}}$ one position to the right, i.e., it becomes a $\sigma_{i+1}^{(-1)^i}$.

\begin{figure}
\labellist
\Large
\pinlabel $n$ at 1900 80
\pinlabel $n$ at 4100 80
\endlabellist
\centering
\includegraphics[height=7cm]{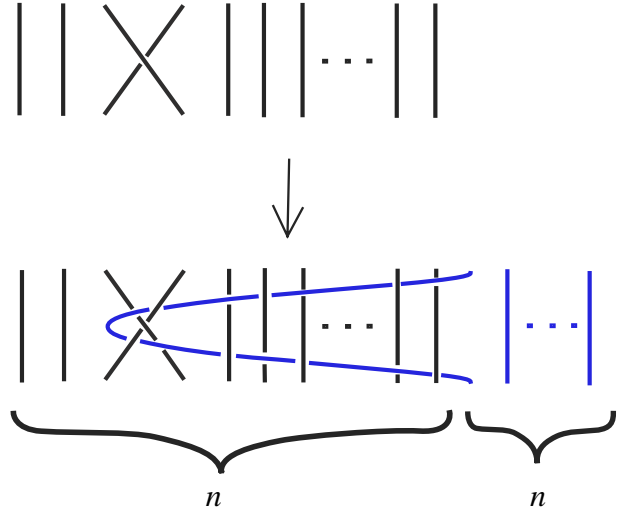}
\caption{The replacement rule for $\sigma_i^{(-1)^{i}}$. We double the number of strands of the braid $B$ and thread the first of the added strands through $B$ to shift the crossing $\sigma_i^{(-1)^i}$ to a position for which it has the desired sign.\label{fig:subhomo}}
\end{figure}

Note that the map $\mathbb{B}_{2n}\to\mathbb{B}_n$ that forgets the strands $1,2,\ldots,n$ maps $B'$ to $B$ and the map that forgets the strands $n+1,n+2,\ldots,2n$ maps $B'$ to the trivial braid. Let $\pi_n:\mathbb{B}_{n}\to S_{n}$ be the permutation representation, where $S_{n}$ is the symmetric group on $n$ elements. Then $\pi_{2n}(B')$ maps $j$ to $\pi_n(B)(j)$ if $j\in\{1,2,\ldots,n\}$ and to $j$ if $j\in\{n+1,n+2,\ldots,2n\}$. Thus the permutation of the strands of $B'$ splits into two disjoint sets of cycles, one involving the strands $j=1,2,\ldots,n$ and the other being the trivial permutation of the remaining strands $j=n+1,n+2,\ldots,2n$.

The braid $B'$ is alternating. This technique of threading an added strand through the braid $B$ to obtain an alternating braid is due to Stallings \cite{stallings}.

We compose $B'$ with $\prod_{k=1}^n\left(\prod_{j=1}^n\sigma_{n+k-j}^{(-1)^{n+k-j+1}}\right)$ (depicted in Figure \ref{fig:alt} for $n=3$) and call the resulting braid $B''$. Note that $B''$ is also an alternating braid and since every generator appears at least once (either with a positive or negative sign), it is homogeneous. 

\begin{figure}
\centering
\includegraphics[height=4cm]{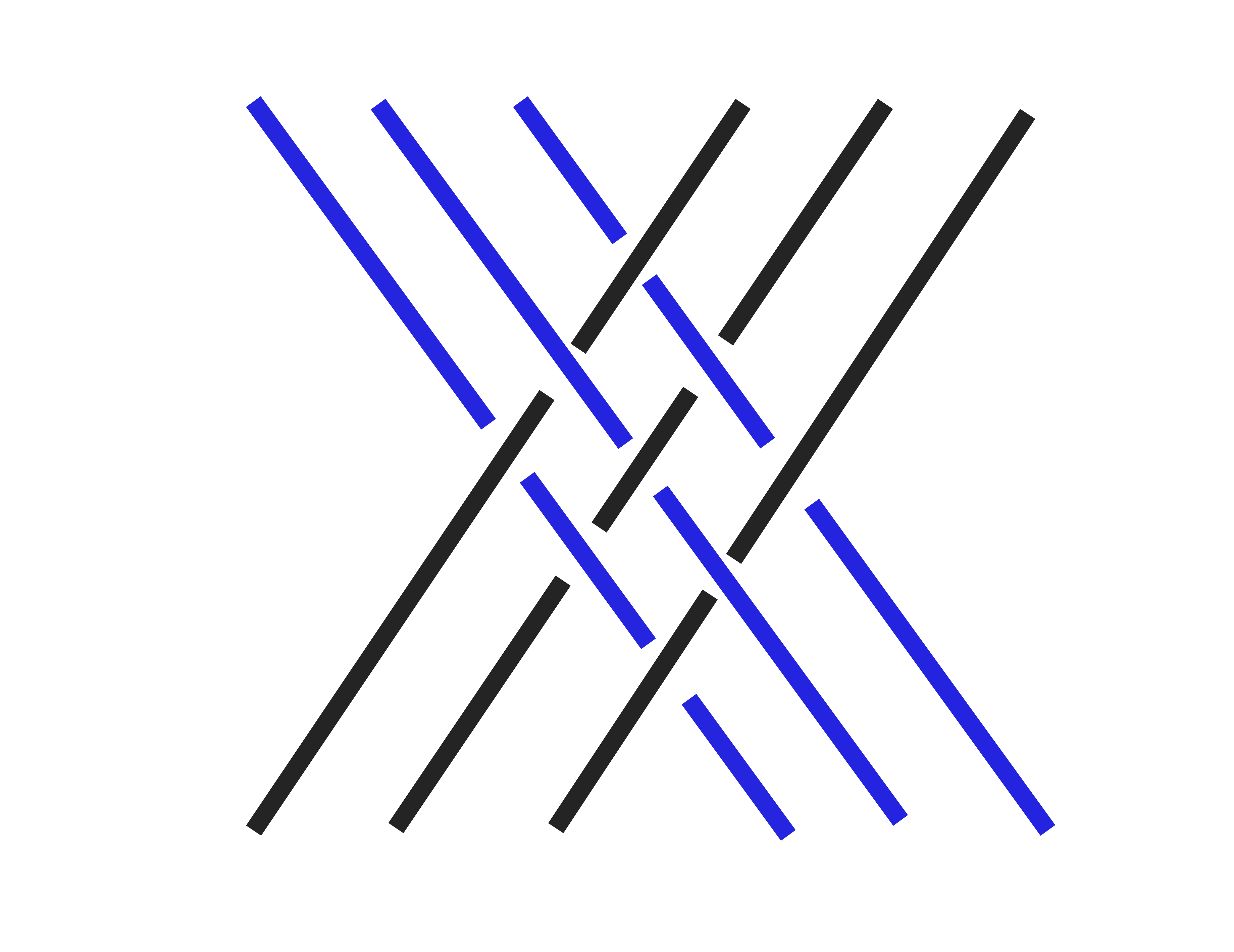}
\caption{The alternating braid $\prod_{k=1}^n\left(\prod_{j=1}^n\sigma_{n+k-j}^{(-1)^{n+k-j+1}}\right)$ for $n=3$. \label{fig:alt}}
\end{figure}
\begin{figure}
\labellist
\Large
\pinlabel a) at 10 1300
\pinlabel b) at 1200 1300
\pinlabel c) at 2800 1300
\pinlabel d) at 1200 -350
\endlabellist
\centering
\includegraphics[height=5cm]{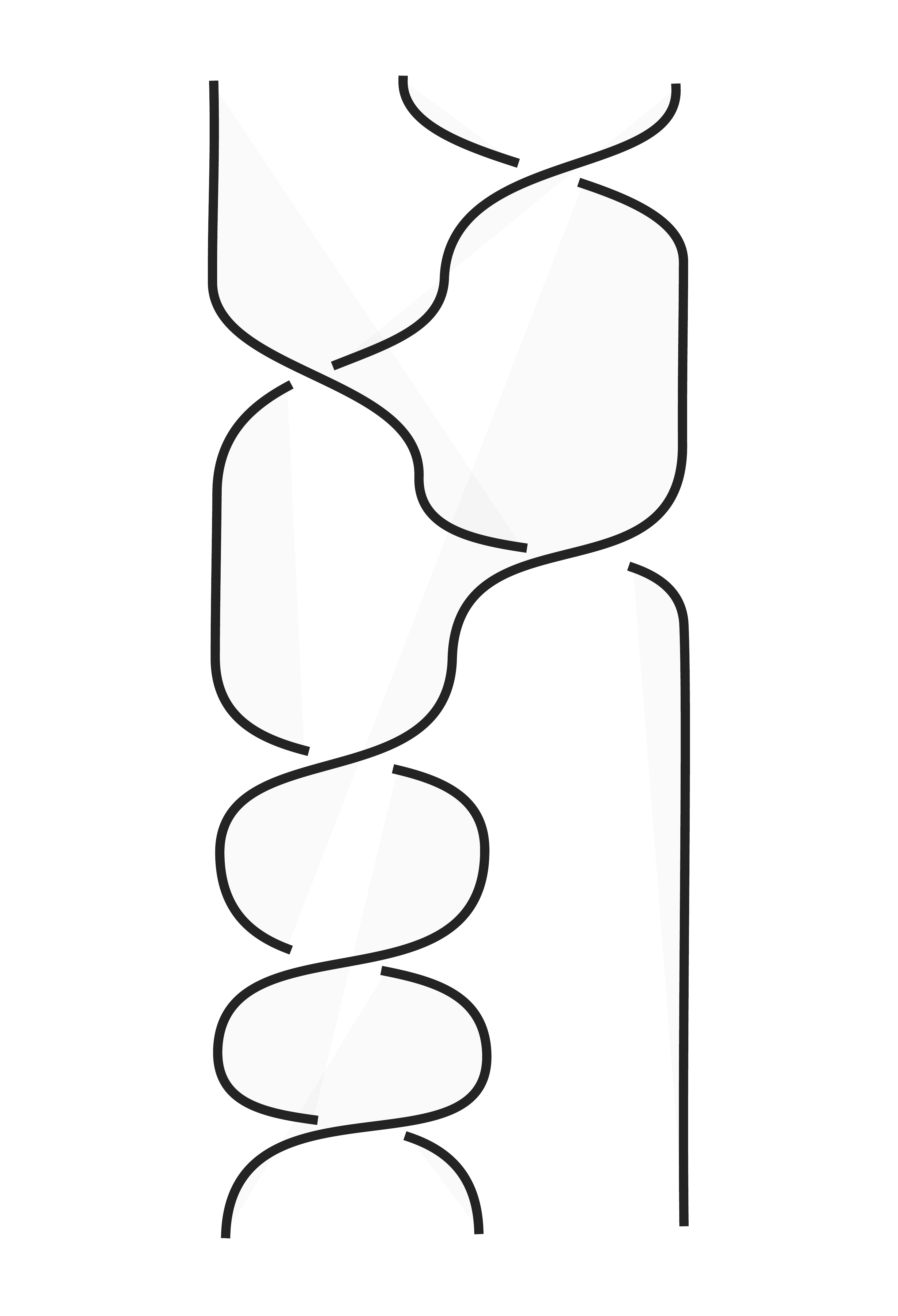}\qquad
\includegraphics[height=6cm]{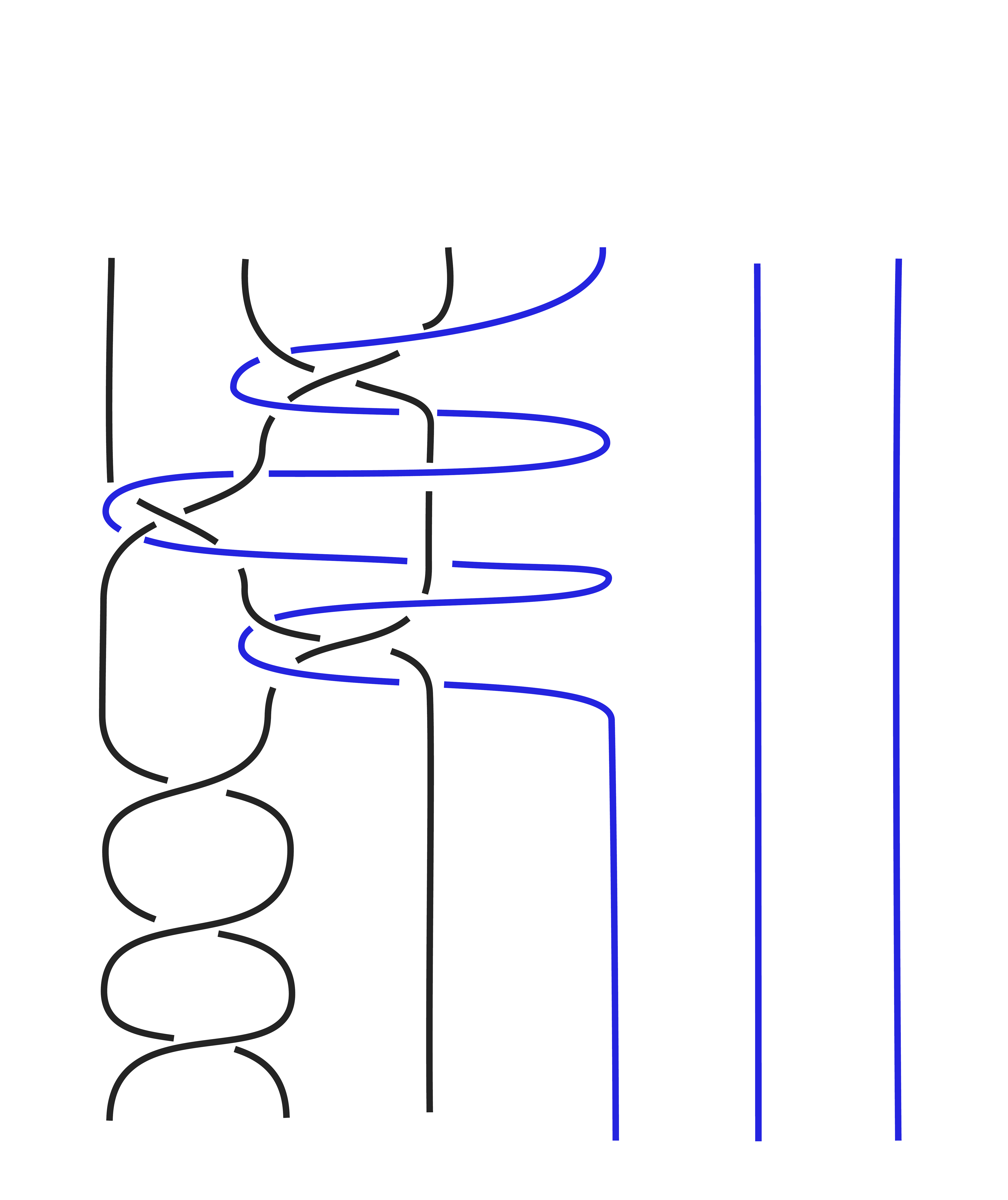}\qquad
\includegraphics[height=5cm]{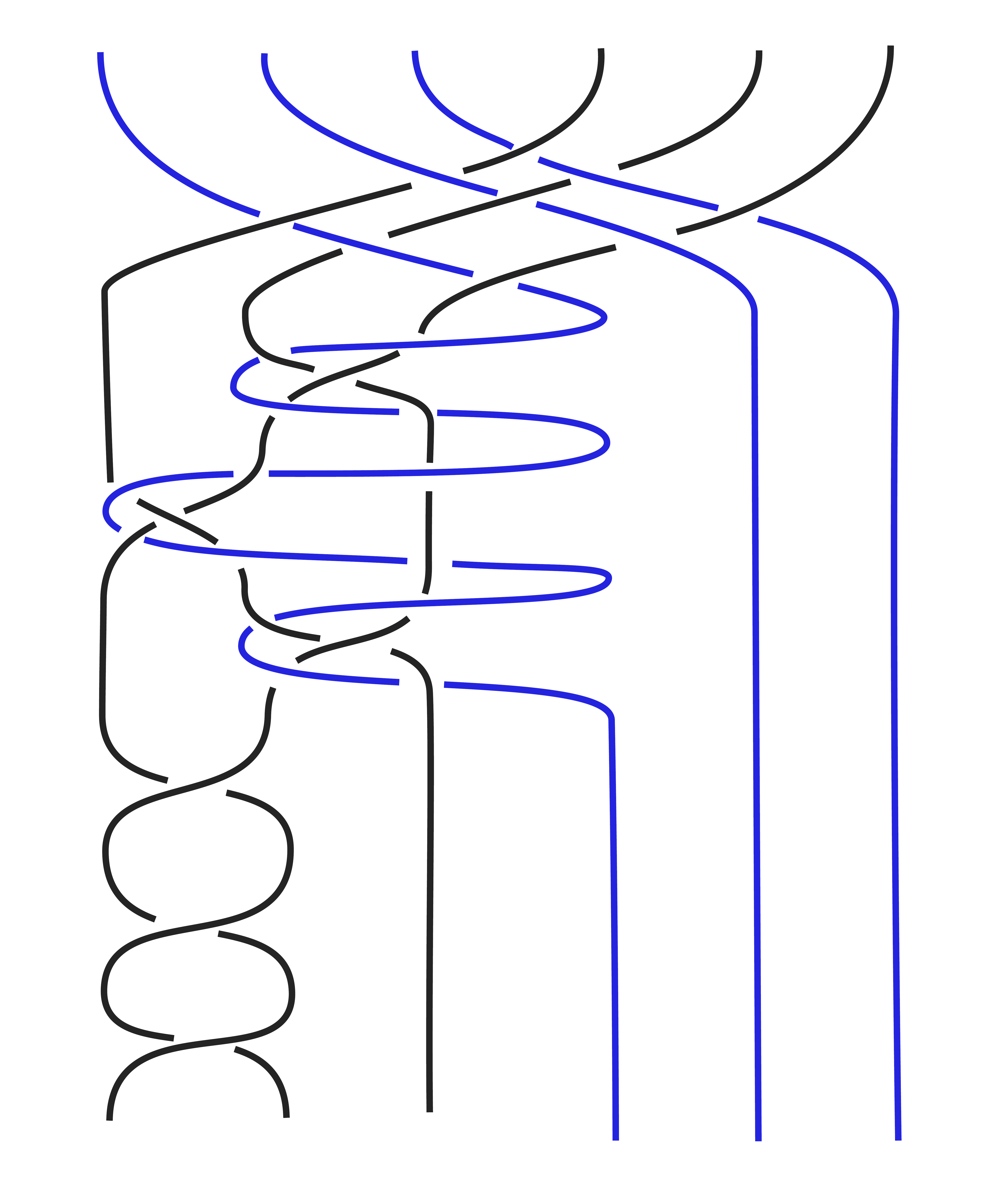}\\
\vspace{1cm}
\includegraphics[height=8cm]{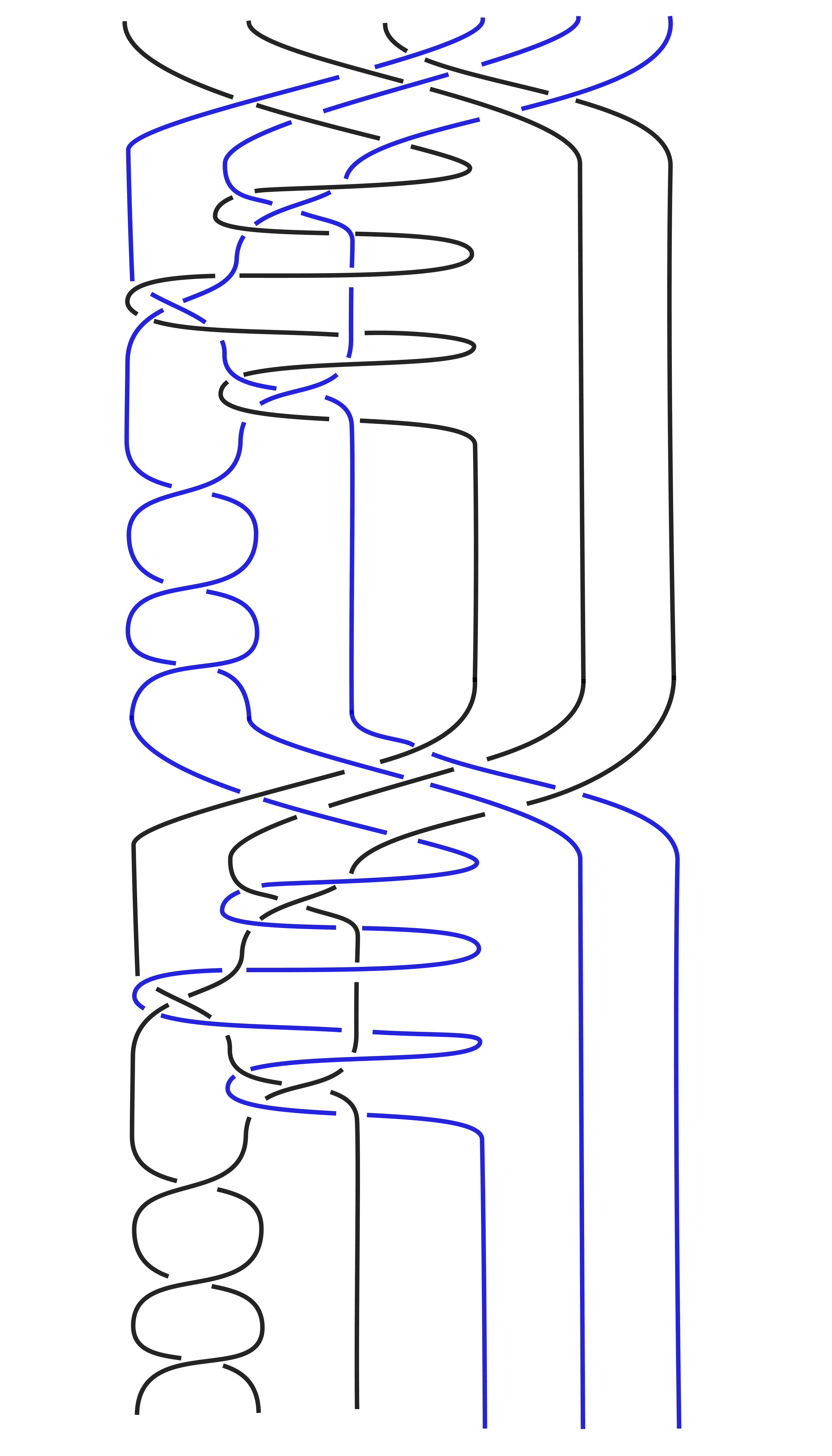}
\caption{The example of the $5_2$-braid. a) A non-homogeneous braid $B$ that closes to the knot $5_2$. b) We double the number of strands and thread the first of the added strands through $B$ to obtain an alternating braid $B'$. c) Composing this with the braid in Figure \ref{fig:alt} yields the alternating (and hence homogeneous) braid $B''$. d) The braid $B''^2$ is the square of a homogeneous braid and its closure consists of two components, each of which give the $5_2$-knot. \label{fig:ex52}}
\end{figure}

Consider now the alternating, homogeneous braid $(B'')^2$. The permutation induced by $(B'')^2$ is given 
\begin{equation}
j\mapsto\begin{cases}
\pi_n(B)(j), &\text{ if }j\in\{1,2,\ldots,n\},\\
n+\pi_n(B)(j-n) &\text{ else.}
\end{cases}
\end{equation}

Note that the map $\mathbb{B}_{2n}\to\mathbb{B}_n$ that forgets the strands $1,2,\ldots,n$ maps $(B'')^2$ to $B$ and the map that forgets the strands $n+1,n+2,\ldots,2n$ maps $(B'')^2$ to $B$ too. Thus the closure of $(B'')^2$, which is a real algebraic link by Theorem \ref{thm:homo} and Theorem \ref{thm:square}, consists of two copies of the closure of $B$, which are linked in some way. In particular, the closure of $B$ is a sublink of the closure of $(B'')^2$.
\end{proof}

Figure \ref{fig:ex52} shows the braids $B$, $B'$, $B''$ and $(B'')^2$ for the example braid $B=\sigma_1^3\sigma_2\sigma_1^{-1}\sigma_2$, which closes to the (non-fibered) knot $5_2$.

Regarding Conjecture \ref{conj:bene} it is reassuring that every link is a sublink of a real algebraic link, given that we already know that every link is a sublink of a fibered link. Theorem \ref{thm:sub} also highlights another stark contrast between algebraic and real algebraic links, since every sublink of an algebraic link is also algebraic, while real algebraic links can have arbitrarily complicated sublinks.

\section{Twisting}
\label{sec:twist}
The positive half-twist of $n$ strand is given by
\begin{equation}
\Delta_n\defeq (\sigma_1\sigma_2\ldots\sigma_{n-1})(\sigma_1\sigma_2\ldots\sigma_{n-2})\ldots(\sigma_1\sigma_2)\sigma_1
\end{equation}
and plays an important role in the braid group. Its square, the full twist $\Delta_n^2$ generates the center of the braid group on $n$ strands and $\Delta_n$ is the key element to defining the Garside Normal Form of a braid. In this form, every braid $B$ on $n$ strands is equivalent to $\Delta_n^{-r}A$ for some $r\in\mathbb{Z}$ and some positive braid $A$ (cf. e.g. \cite{algo}). It follows that any braid can be turned into a positive braid by adding a certain number of positive half-twists to it. Since positive braid closures are fibered, there is for every braid $B$ on $n$ strands a lowest exponent $k$ such that $\Delta_n^{2k} B$ is fibered. It is not too difficult to see that $k$ can be taken to be less than $\tfrac{k_-+1}{2}$, where $k_-$ is the number of negative crossings of $B$ \cite{algo}. In this section we use the theory of P-fibered braids to prove Theorem \ref{thm:twist}, which establishes a better bound on $k$.




Let
\begin{equation}
Y_{i}=\sigma_{i}^{-1}\sigma_{i-1}^{-1}\ldots\sigma_2^{-1}\sigma_1^2\sigma_2\sigma_3\ldots\sigma_{i}\qquad\text{ if }i\geq2
\end{equation}
and $Y_1=\sigma_1^2$ be braids on $n$ strands. Furthermore, let 
\begin{equation}
X_i=\begin{cases}
Y_{\tfrac{i+1}{2}} &\text{ if }i\text{ is odd,}\\
Y_{\tfrac{i}{2}+\lfloor\tfrac{s-1}{2}\rfloor}&\text{ if }i\text{ is even}.
\end{cases}
\end{equation}

An example of such a braid on four strands, namely $Y_3=X_2$, is shown in Figure \ref{fig:rudolph}a). The corresponding motion of points in the complex plane is shown in Figure \ref{fig:rudolph}b). Note that these braids can be parametrised such that only one point is non-stationary. It moves in front of other strands, makes one twist around the first strand and comes back to its original starting position.

\begin{proposition}[cf. \cite{bode:real, adicact, bode:review, rudolph}]
\label{rudolph}
For every parametrisation $(0,v_1(t),v_2(t),\ldots,v_{n-1}(t))$ of any braid of the form $\prod_{j=1}^k X_{i_j}^{\epsilon_j}$, $\epsilon_j\in\{\pm 1\}$, there is a loop in the space of complex polynomials $g_t$, $t\in S^1$, such that the roots of $g(u,t)=g_t(u)$ form the braid $\prod_{j=1}^k \sigma_{i_j}^{\epsilon_j}$ and the critical values of $g_t$ are $v_1(t),v_2(t),\ldots,v_{n-1}(t)$.
\end{proposition}

\begin{figure}
\labellist
\Large
\pinlabel a) at 10 1300
\pinlabel b) at 1200 1300
\pinlabel c) at 10 -50
\pinlabel d) at 1800 -50
\endlabellist
\centering
\includegraphics[height=4cm]{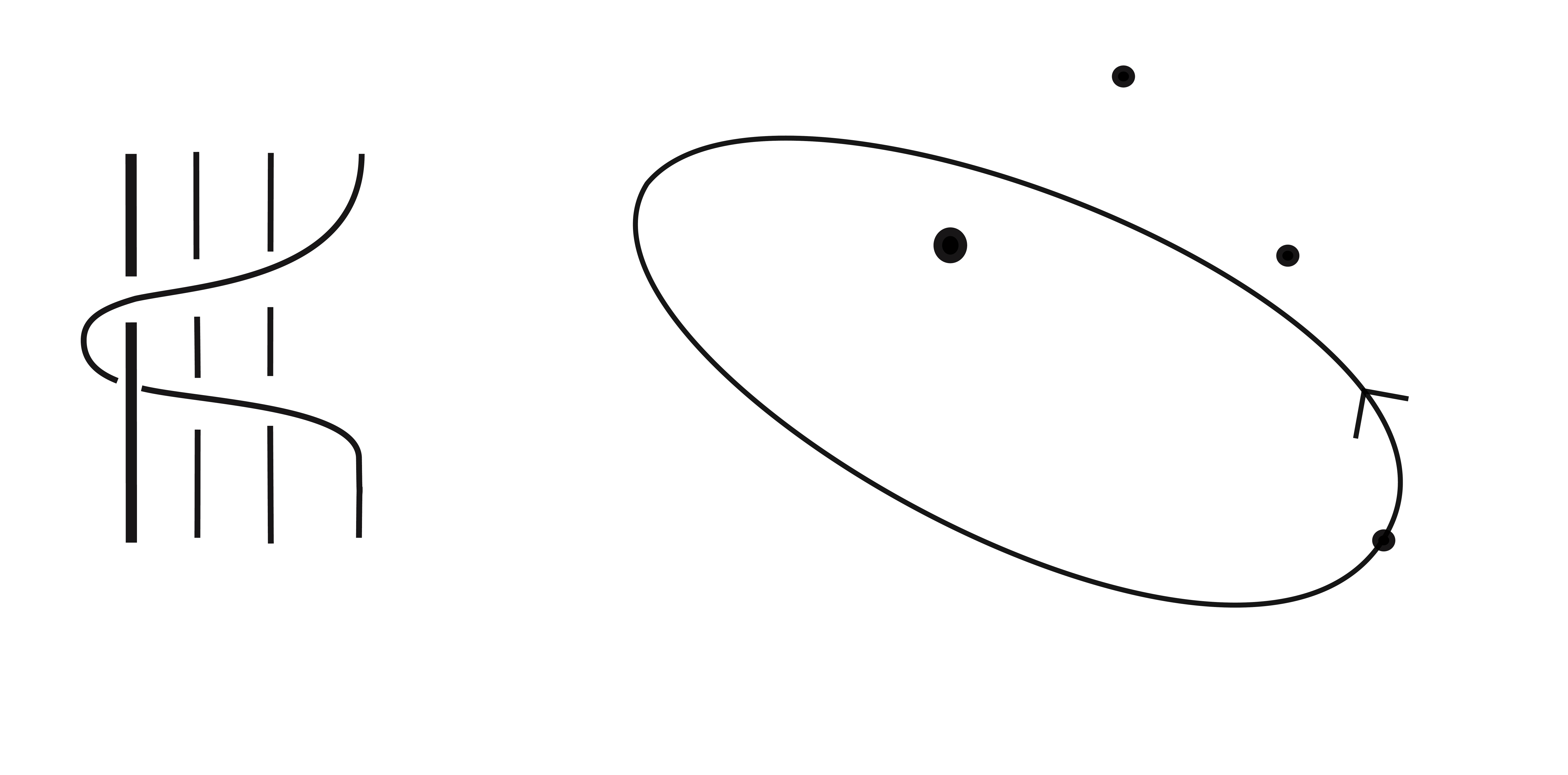}
\includegraphics[height=6cm]{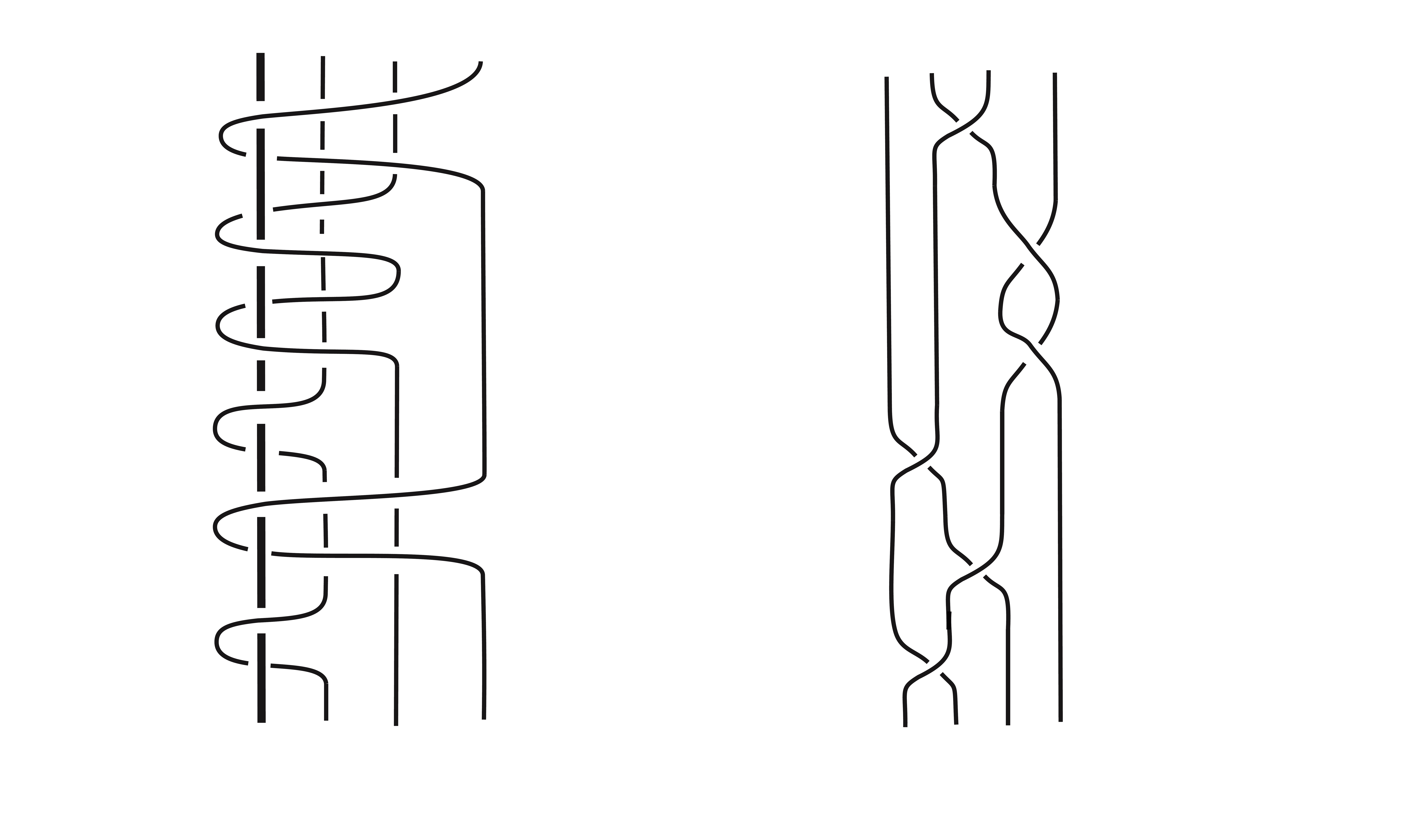}
\caption{a) The braid $Y_3$ on 4 strands which equals $X_2$. b) The same braid as a motion of points in the complex plane. c) The braid $X_1X_2X_1X_3^{-2}X_2$. As a geometric braid it can be regarded as the union of the critical values of a loop of polynomials $g_t$ and the 0-strand $(0,t)\subset\mathbb{C}\times[0,2\pi]$, $t\in[0,2\pi]$. d) The roots of $g_t$ can be taken to form the braid $\sigma_1\sigma_2\sigma_1\sigma_3^{-2}\sigma_2$. \label{fig:rudolph}}
\end{figure}

In \cite{bode:review}, Proposition \ref{rudolph} was stated differently, namely that the roots of $g_t$ form a conjugate of $\prod_{j=1}^k \sigma_{i_j}^{\epsilon_j}$. However, a closer look at the proofs in \cite{bode:review} quickly establishes that the element that we are conjugating by is the trivial braid. 


\begin{proof}[Proof of Theorem \ref{thm:twist}]
Note that we can parametrise each $Y_i$ as $(0,v_1(t),v_2(t),\ldots,v_{n-1}(t))$ such that $v_j(t)$ is constant for all $j\neq i$ and $v_i(t)$ moves counter-clockwise on an ellipse around the origin. In particular, we can choose $\tfrac{\partial \arg v_i}{\partial t}=1$. This leads to a parametrisation of $\prod_{j=1}^k X_{i_j}^{\epsilon_j}$ by composition of these parametrisations or their inverses. We can parametrise this composition such that the parts that correspond to a positive sign, i.e., $\sigma_i$ and $X_i$ rather than $\sigma_i^{-1}$ and $X_i^{-1}$, are executed in an $\varepsilon$-amount of time for an arbitrarily small $\varepsilon>0$, so that the corresponding $\tfrac{\partial \arg v_i}{\partial t}$ becomes very large. For a negative sign, i.e., $\sigma_i^{-1}$ and $X_i^{-1}$, the corresponding critical value moves clockwise. However, for each part that corresponds to a negative crossing $\sigma_i^{-1}$, the corresponding critical value $v_i(t)$ has $\tfrac{2\pi-\varepsilon}{k_{-}}$ amount of time to complete one clockwise turn around its ellipse, where $k_{-}$ is the number of negative crossings in $\prod_{j=1}^k \sigma_{i_j}^{\epsilon_j}$. This means we have a parametrisation of $\prod_{j=1}^k X_{i_j}^{\epsilon_j}$ with the property that
\begin{equation}
\label{eq:lowerbound}
\underset{t\in[0,2\pi]}{\underset{q=1,2,\ldots,s-1}{\min}} \frac{\partial\arg v_q(t)}{\partial t}\geq -\frac{2\pi}{2\pi-\varepsilon}k_{-}> -k_{-}-1.
\end{equation} 

By Proposition \ref{rudolph} this parametrisation of $\prod_{j=1}^k X_{i_j}^{\epsilon_j}$ corresponds to the strand $(0,t)\subset\mathbb{C}\times[0,2\pi]$ and the critical values $(v_1(t),v_2(t),\ldots,v_{n-1}(t))$ of a parametrised family of polynomials $g_t$, $t\in[0,2\pi]$, whose roots form the desired braid $B=\prod_{j=1}^k \sigma_{i_j}^{\epsilon_j}$.

We use the same notation for $g_t$ as in the previous sections, that is, the $j$th strand of the $i$th component of the closure of $B$ is parametrised by $(z_{i,j}(t),t)\subset\mathbb{C}\times[0,2\pi]$ for a smooth function $z_{i,j}(t):[0,2\pi]\to\mathbb{C}$. 


Consider now the braid $\Delta^{2k}B$, parametrised by $(\rme^{\rmi kt}z_{i,j}(t),t)\subset\mathbb{C}\times[0,2\pi]$, $i=1,2,\ldots,m$, $j=1,2,\ldots,n_i$, for some $k\in\mathbb{N}$. 

Note that the critical points of $\widetilde{g}_t(u)=\prod_{i=1}^{m}\prod_{j=1}^{n_i}(u-\rme^{\rmi kt}z_{i,j}(t))$ are given by $\rme^{\rmi kt}c_q(t)$, $q=1,2,\ldots,n-1$, where $c_q(t)$, $q=1,2,\ldots,n-1$, are the critical points of $g_t(u)=\prod_{i=1}^m\prod_{j=1}^{n_i}(u-z_{i,j}(t))$. Hence the critical values $v'_q(t)$ of $\widetilde{g}_t$ are given by $\rme^{\rmi nk t}v_q(t)$, $q=1,2,\ldots,n-1$.

It follows from Equation (\ref{eq:lowerbound}) that the critical values $v'_q(t)$, $q=1,2,\ldots,n-1$, of $\widetilde{g}_t$ satisfy
\begin{equation}
\label{eq:lowerbound2}
\underset{t\in[0,2\pi]}{\underset{q=1,2,\ldots,s-1}{\min}} \frac{\partial\arg v'_q(t)}{\partial t}\geq -\frac{2\pi}{2\pi-\varepsilon}k_{-}+kn> -k_{-}-1+kn.
\end{equation} 

In particular, if $k\geq \tfrac{k_{-}+1}{n}$ we have constructed a parametrisation of $\Delta^{2k}B$ as a P-fibered braid.
\end{proof}




\end{document}